\documentclass[12pt,a4paper]{article}

\usepackage[a4paper,left=2.5cm,right=2.5cm,
top=2.5cm,bottom=4.5cm]{geometry}

\usepackage{amsmath,amssymb,amsfonts,amsthm}
\usepackage{mathrsfs}
\usepackage{graphicx}
\usepackage{booktabs}
\usepackage{multirow}
\usepackage{xcolor}
\usepackage{enumerate}
\usepackage{appendix}
\usepackage{hyperref}

\newtheorem{theorem}{Theorem}[section]

\newtheorem{lemma}[theorem]{Lemma}

\theoremstyle{definition}

\theoremstyle{remark}

\title{\bfseries
Optimal inequalities involving Casorati curvatures along
Riemannian maps and Riemannian submersions for quaternionic space form
}

\author{
Ravindra Singh\\[2mm]
\small Department of Mathematics, Banaras Hindu University\\
\small Varanasi 221005, India\\
\small Email: \texttt{khandelrs@bhu.ac.in}\\
\small ORCID: 0009-0009-1270-3831
}

\date{}

\begin{document}
\maketitle

\begin{abstract}
\noindent In this paper, we establish Casorati inequalities for Riemannian maps and Riemannian submersions involving quaternionic space forms, and we provide
geometric characterisations of their equality cases. First, we derive
Casorati inequalities for Riemannian maps to quaternionic space forms and
describe the corresponding equality cases, showing that the leaves of the
range space are invariantly quasi-umbilical, and that the associated
shape operator matrix commutes. Next, we obtain Casorati inequalities involving the fundamental tensor fields $T$ and $A$ for Riemannian submersions from quaternionic space forms onto Riemannian manifolds, together with their geometric interpretations. In particular, we prove that the equality case corresponding to the tensor field $A$ along the horizontal distribution is equivalent to the integrability of the horizontal distribution. Moreover, the equality case associated with the tensor field $T$ along the vertical distribution characterises fibres that are invariantly quasi-umbilical with a commuting shape operator matrix. Finally, the simultaneous equality cases involving both tensor fields $T$ and $A$ along the horizontal and vertical distributions imply the integrability of the horizontal distribution together with the invariantly quasi-umbilical nature of the fibres and the commutativity of the corresponding shape operators.
\end{abstract}

\medskip
\noindent\textbf{Keywords:}
Riemannian manifolds; space forms; Casorati curvature;
Riemannian submersions; Riemannian maps.

\medskip
\noindent\textbf{MSC (2020):}
53B20, 53B35, 53C15, 53D15.

\section{Introduction} \label{sec1}
Gaussian curvature is an essential tool in differential geometry that describes the geometry and shape of surfaces as well as their behavior at particular points. However, Gaussian curvature may vanish even on surfaces that are curved. To overcome this drawback, Casorati \cite{Casorati_1890} introduced an extrinsic invariant, known as
{\em Casorati curvature}, for regular surfaces. This curvature disappears only
at planar points and is useful for visualising shapes and appearances
\cite{DHV_2008, OV_2011}. Thereafter, many authors obtained inequalities involving
Casorati curvature, which nowadays are known as {\em Casorati inequalities}, for
submanifolds of different ambient spaces (see
\cite{ALVY, ASJ, DHV_2008, Ghisoiu, Lone_2017a, Vilcu, LLV_2017, LLV_2020, LLV_2022,
Lone_2017, Lone_2019, Lone_2019a, Siddiqui_2018, Zhang_Zhang, Zhang_Pan_Zhang}).
For more details, we refer the reader to \cite{Chen_Survey}.

Riemannian submersions are smooth mappings that play an important role in
various areas of mathematics and theoretical physics, including mechanics,
relativity, spacetime geometry, robotics, supergravity, superstring theory,
Kaluza-Klein theory, and Yang–Mills theory. The notion of a Riemannian map,
introduced by Fischer \cite{Fischer_1992}, naturally extends the concepts of
submanifolds and Riemannian submersions. Riemannian maps possess rich geometric
structures and admit several applications \cite{GRK_book, Sahin_book}. In
particular, they satisfy the eikonal equation, which serves as a link between
geometric optics and physical optics. Moreover, Riemannian maps provide a
useful framework for comparing the geometric properties of the source and
target Riemannian manifolds. Owing to these appealing features, Riemannian
submersions and Riemannian maps have been extensively studied in the presence
of various geometric structures (see \cite{Falcitelli_2004, Sahin_book, Yano}
and the references therein).

Casorati inequalities for Riemannian submersions and Riemannian maps were first
established for real and complex space forms in \cite{LLSV}, and were later
extended to Sasakian space forms in \cite{PLS}. More recently, Singh
{\em et al.} \cite{SMM_2026} derived Casorati inequalities for Riemannian
submersions and Riemannian maps for real, complex, real K\"ahler, generalized
complex, Sasakian, Kenmotsu, cosymplectic, almost $C(\alpha)$, and generalized
Sasakian space forms, along both the vertical and horizontal distributions.
Motivated by these developments, Singh \cite{Singh_2026} further obtained a
Casorati inequality for Riemannian submersions whose total spaces belong to
the above-mentioned classes of space forms, which simultaneously involves
the Casorati curvatures of both the vertical and horizontal distributions. Motivated by the above progress, in the present paper, we establish Casorati
inequalities for Riemannian submersions and Riemannian maps for quaternionic space forms.
These results indicate that Casorati inequalities may provide a natural
link between geometry and mathematical physics. 

This paper is organised as follows. In Section~\ref{sec2}, we present the definition of quaternionic space form and the necessary lemma. In Section~\ref{sec3}, we obtain Casorati inequalities for Riemannian maps to quaternionic space forms. In Section~\ref{sec4}, we present the fundamental tools for Riemannian submersions. In subsection~\ref{sec5}, we derive Casorati inequalities for Riemannian submersions whose total spaces are quaternionic
space forms along the vertical distribution. In subsection~\ref{sec6}, we obtain
Casorati inequalities for Riemannian submersions whose total spaces are
quaternionic space forms along the horizontal distribution. In subection~\ref{sec7}, finally, we establish Casorati inequalities for Riemannian submersions whose
total spaces are quaternionic space forms along both the horizontal and vertical distributions.

\section{Preliminaries} \label{sec2}
\subsection{Quaternionic Space Forms}
Let $\left( M,g\right) $ be a $4m$-dimensional Riemannian manifold equipped
with a $3$-dimensional vector bundle $\gamma $ of tensors of type $\left(
1,1\right) $ with a local basis formed by Hermitian structures $\left\{
J_{1},J_{2},J_{3}\right\} $ such that%
\[
J_{1}\circ J_{2}=-J_{2}\circ J_{1}=J_{3}\quad {\rm and\quad }\nabla
_{X}J_{\alpha }=\sum_{\beta =1}^{3}Q_{\alpha \beta }J_{\beta },\quad \alpha
\in \left\{ 1,2,3\right\} 
\]%
for any vector field $X$, where $\nabla $ is the Levi-Civita connection of $%
g $ and $Q_{\alpha \beta }$ are certain local $1$-forms on $M$ such that $%
Q_{\alpha \beta }+Q_{\beta \alpha }=0$. Then $\left( g,\gamma \right) $ is
said to be a quaternionic Kaehler structure on $M$ and $\left( M,g,\gamma
\right) $ is said to be a quaternionic Kaehler manifold.

Let $\left( M,g,\gamma \right) $ be a quaternionic Kaehler manifold and let $%
X$ be a non null vector field on $M$. Then the $4$-dimensional plane $%
Q\left( X\right) $, spanned by $\left\{ X,J_{1}X,J_{2}X,J_{3}X\right\} $, is
called a quaternionic $4$-plane. Any $2$-plane in $Q\left( X\right) $ is
called a quaternionic plane. The sectional curvature of a quaternionic plane
is called a quaternionic sectional curvature. If the quaternionic sectional
curvatures of a quaternionic Kaehler manifold $\left( M,g,\gamma \right) $
are equal to a real constant $c$, then it is said to be a quaternionic space
form, and is denoted by $M\left( c\right) $. It is well known that a
quaternionic Kaehler manifold is a quaternionic space form $M\left( c\right) 
$ if and only if its Riemann curvature tensor is given by 
\begin{align}
& R\left( {\cal Z}_{1},{\cal Z}_{2},{\cal Z}_{3},{\cal Z}_{4}\right) 
\nonumber \\
& =\frac{c}{4}\{g({\cal Z}_{2},{\cal Z}_{3})g({\cal Z}_{1},{\cal Z}_{4})-g(%
{\cal Z}_{1},{\cal Z}_{3})g({\cal Z}_{2},{\cal Z}_{4})\}  \nonumber \\
& +\frac{c}{4}\sum_{\alpha =1}^{3}\{g({\cal Z}_{1},J_{\alpha }{\cal Z}%
_{3})g\left( J_{\alpha }{\cal Z}_{2},{\cal Z}_{4}\right) -g({\cal Z}%
_{2},J_{\alpha }{\cal Z}_{3})g\left( J_{\alpha }{\cal Z}_{1},{\cal Z}%
_{4}\right) +2g({\cal Z}_{1},J_{\alpha }{\cal Z}_{2})g\left( J_{\alpha }%
{\cal Z}_{3},{\cal Z}_{4}\right) \}.  \label{eq-QSF}
\end{align}%
for all vector fields $X$, $Y$, $Z$, $W$ on $M$ and any local basis $\left\{
J_{1},J_{2},J_{3}\right\} $. For details we refer to \cite{Ishihara_1974}.

\begin{enumerate}
\item In case of Riemannian map, for any ${\cal Z}\in \Gamma (TM_{2})$, we
write 
\begin{equation}
J_{\alpha }{\cal Z}=P_{\alpha }^{{\cal R}}{\cal Z}+Q_{\alpha }^{{\cal R}^{\perp}}%
{\cal Z},  \label{decompose_GCSF_RM}
\end{equation}%
where $P_{\alpha }^{{\cal R}}{\cal Z}\in \Gamma ({\rm range~}F_{\ast })$, $%
Q_{\alpha }^{{\cal R}^{\perp}}{\cal Z}\in \Gamma ({\rm range~}F_{\ast })^{\perp }$
such that 
\begin{equation}
\Vert P_{\alpha }^{{\cal R}}\Vert ^{2}=\sum\limits_{i,j=1}^{s}\left(
g_{2}(F_{\ast }h_{i},J_{\alpha }F_{\ast }h_{j})\right)
^{2}=\sum\limits_{i,j=1}^{r}\left( g_{2}(F_{\ast }h_{i},P_{\alpha }^{{\cal R}%
}F_{\ast }h_{j})\right) ^{2},  \label{eq-PRalpha}
\end{equation}
where $\{h_{1},\ldots,h_{s}\}$ and $\{F_{\ast}h_{1},\ldots,F_{\ast}h_{s}\}$ are bases of $(\ker F_{\ast})^{\perp}$ and ${\rm range}F_{\ast}$, respectively.

\item In the case of Riemannian submersion for any ${\cal Z}\in \Gamma
(TM_{1})$, we write 
\begin{equation}
J_{\alpha }{\cal Z}=P_{\alpha }{\cal Z}+{Q_{\alpha }}{\cal Z},
\label{decompose_GCSF_VRS}
\end{equation}%
where {$P_{\alpha }$}${\cal Z}\in \Gamma ({\rm ker~}F_{\ast })^{\perp }$, {$%
Q_{\alpha }$}${\cal Z}\in \Gamma ({\rm ker~}F_{\ast })$, such that $\Vert ${$%
P_{\alpha }$}$\Vert ^{2}=\sum\limits_{i,j=1}^{s}\left( g_{1}(h_{i},J_{\alpha
}h_{j})\right) ^{2}=\sum\limits_{i,j=1}^{s}\left( g_{1}(h_{i},{P_{\alpha }}%
h_{j})\right) ^{2}$, $\Vert ${$Q_{\alpha }$}$\Vert
^{2}=\sum\limits_{i,j=1}^{\ell }\left( g_{1}(v_{i},J_{\alpha }v_{j})\right)
^{2}=\sum\limits_{i,j=1}^{r}\left( g_{1}(v_{i},{Q_{\alpha }}v_{j})\right)
^{2}$, and $\Vert ${$P_{\alpha }^{{\cal V}}$}$\Vert
^{2}=\sum\limits_{j=1}^{\ell }\sum\limits_{i=1}^{s}\left(
g_{1}(h_{i},P_{\alpha }v_{j})\right) ^{2}$,
\end{enumerate}
where $\{h_{1},\ldots,h_{s}\}$ and $\{v_{1},\ldots,v_{\ell}\}$ are bases of $(\ker F_{\ast})^{\perp}$ and $(\ker F_{\ast})$, respectively.

\begin{lemma}
\label{Lemma_Tripathi} $\cite{Tripathi_2017}$ Let $\Lambda =\{(t_{1},\dots
,t_{n})\in {\Bbb R}^{n}:t_{1}+\cdots +t_{n}=k\}$ be a hyperplane of ${\Bbb R}%
^{n}$, and $f:{\Bbb R}^{n}\rightarrow {\Bbb R}$ be a quadratic form given by 
\[
f\left( t_{1},\dots ,t_{n}\right) =\lambda
_{1}\sum_{i=1}^{n-1}t_{i}^{2}+\lambda _{2}t_{n}^{2}-2\sum_{1\leq i<j\leq
n}t_{i}t_{j},\quad \lambda _{1},\lambda _{2}\in {\Bbb R}^{+}. 
\]%
Then the constrained extremum problem $\min\limits_{(t_{1},\dots ,t_{n})\in
\Lambda }f$ has the global solution $t_{1}=t_{2}=\cdots =t_{n-1}=\frac{k}{%
\lambda _{1}+1},\quad t_{n}=\frac{k}{\lambda _{2}+1}=\frac{k\left(
n-1\right) }{\left( \lambda _{1}+1\right) \lambda _{2}}=\frac{k\left(
\lambda _{1}-n+2\right) }{\lambda _{1}+1}$, provided that $\lambda _{2}=%
\frac{n-1}{\lambda _{1}-n+2}$.
\end{lemma}

\section{Casorati inequalities for quaternionic space form to\ Riemannian
maps} \label{sec3}

\label{sec_2}In this section we derive optimal inequalities involving
Casorati curvature for Riemannian map whose target space form is
quaternionic space form, discuss their equality cases, which shows
invariantly quasi umblicity of leaves of range spaces and shape operator matrix commutes.

Let $\pi :(N_{1},g_{1})\rightarrow (N_{2},g_{2})$ be a Riemannian map
between two Riemannian manifolds, where $\dim N_{1}=n$, $\dim N_{2}=n_{2}$.
For each pont $p\in N_{1}$, its derivative map $\pi _{\ast
p}:T_{p}N_{1}\rightarrow T_{\pi (p)}N_{2}$, and tangent space $T_{p}N_{1}$
can be decomposed as direct sum of two orthogonal distribution $(\ker \pi
_{\ast p})$ and $(\ker \pi _{\ast p})^{\perp }$ such that

\[
T_{p}N_{1}=(\ker \pi _{\ast p})\oplus (\ker \pi _{\ast p})^{\perp }. 
\]%
Similarly we decompose at $\pi (p)\in N_{2}$, 
\[
T_{\pi (p)}N_{2}=({\rm range}{\text{ }}\pi _{\ast p})\oplus ({\rm range}{%
\text{ }}\pi _{\ast p})^{\perp }. 
\]%
For all $X,Y\in \Gamma (\ker \pi _{\ast })^{\perp }$ the following
equation holds \cite{Fischer_1992}: 
\begin{equation}
g_{1}(X,Y)=g_{2}(\pi _{\ast }X,\pi _{\ast }Y).  \label{hor_range_isometry}
\end{equation}

\noindent The bundle {\rm Hom}$\left( T{N_{1}},\pi ^{-1}T{N_{2}}\right) $
admits an induced connection $\nabla $ from the Levi-Civita connection $%
\nabla ^{{N_{1}}}$ on {$N$}${_{1}}$. Then symmetric {\it second fundamental
form} of $\pi $ is given by \cite{Nore_1986} 
\[
\left( \nabla \pi _{\ast }\right) \left( X,Y\right) =\nabla _{\pi_{\ast}X}^{N_{2}
}\pi _{\ast }(Y)-\pi _{\ast }\left( \nabla _{X}^{{N_{1}}}Y\right) 
\]%
for $X,Y\in \Gamma \left(TN_{1}\right)$, where $\nabla
^{M_{1}}$ and $\nabla^{M_{2}}$ is the Levi-Civita connections on $M_{1}$ and ${M_{2}}$, respectively. In addition, by \cite{Sahin_2010} $%
\left( \nabla \pi _{\ast }\right) \left( X,Y\right) $ is completely
contained in $\left( {\rm range~}\pi _{\ast }\right) ^{\perp }${\rm \ }for
all$\ X,Y\in \Gamma \left( \ker \pi _{\ast }\right) ^{\perp }$. Furthermore,
the Gauss equation for $\pi $ is defined as \cite[p. 189]{Sahin_book} 
{\small 
\begin{eqnarray}
g_{2}\left( R^{{N_{2}}}\left( \pi _{\ast }W_{1},\pi _{\ast }W_{2}\right) \pi
_{\ast }W_{3},\pi _{\ast }W_{4}\right) &=&g_{1}\left( R^{{N_{1}}}\left(
W_{1},W_{2}\right) W_{3},W_{4}\right)  \nonumber \\
&&+g_{2}\left( \left( \nabla \pi _{\ast }\right) \left( W_{1},W_{3}\right)
,\left( \nabla \pi _{\ast }\right) \left( W_{2},W_{4}\right) \right) 
\nonumber \\
&&-g_{2}\left( \left( \nabla \pi _{\ast }\right) \left( W_{1},W_{4}\right)
,\left( \nabla \pi _{\ast }\right) \left( W_{2},W_{3}\right) \right) ,
\label{Gauss_Codazzi_RM}
\end{eqnarray}%
} where $W_{i}\in \Gamma \left( \ker \pi _{\ast }\right) ^{\perp }$. Here, $%
R^{{N_{1}}}$ is the curvature tensor of $N_{1}$ and $R^{{N_{2}}}$ be the
curvature tensors of {$N$}${_{2}}$. \newline

Let $\left\{ h_{1},\ldots ,h_{s}\right\} $ be a orthonormal basis of $\left(
\ker F_{\ast p}\right) ^{\perp }$ and $\left\{ h_{s+1},\ldots
,h_{n_{1}}\right\} $ be a orthonormal basis of $\left( \ker F_{\ast
p}\right) $, then the scalar curvatures defined by 
\begin{equation}
2{\tau }_{N_{1}}^{{\cal H}}=\sum\limits_{i,j=1}^{s}g_{1}\left( R^{{N_{1}}%
}(h_{i},h_{j})h_{j},h_{i}\right) ,\ 2{\tau }_{N_{2}}^{{\cal R}%
}=\sum\limits_{i,j=1}^{s}g_{2}\left( R^{{N_{2}}}(F_{\ast }h_{i},F_{\ast
}h_{j})F_{\ast }h_{j},F_{\ast }h_{i}\right)  \label{eq-scalH}
\end{equation}%
Consequently, we define the normalized scalar curvatures for horizontal
space 
\begin{equation}
\rho ^{{\cal H}}=\frac{2{\tau }_{N_{1}}^{{\cal H}}}{s\left( s-1\right) }%
,\quad \rho ^{{\cal R}}=\frac{2{\tau }_{N_{2}}^{{\cal R}}}{s\left(
s-1\right) }.  \label{eq-NSCALHR}
\end{equation}%
Supposing $\{V_{s+1},\dots ,V_{n_{2}}\}$ an orthonormal basis of $\left( 
{\rm range~}\pi _{\ast }\right) ^{\perp }$ we set, 
\begin{eqnarray*}
B_{ij}^{{\cal H}^{\alpha }} &=&g_{2}\left( (\nabla \pi _{\ast
})(h_{i},h_{j}),V_{\alpha }\right) ,\quad i,j=1,\dots ,s,\quad \alpha
=s+1,\dots ,n_{2}, \\
\left\Vert B^{{\cal H}}\right\Vert ^{2} &=&\sum_{i,j=1}^{s}g_{2}\left(
(\nabla \pi _{\ast })(h_{i},h_{j}),(\nabla \pi _{\ast })(h_{i},h_{j})\right)
, \\
{\rm trace\,}B^{{\cal H}} &=&\sum_{i=1}^{s}(\nabla \pi _{\ast })\left(
h_{i},h_{i}\right) , \\
\left\Vert {\rm trace\,}B^{{\cal H}}\right\Vert ^{2} &=&g_{2}\left( {\rm %
trace\,}B^{{\cal H}},{\rm trace\,}B^{{\cal H}}\right) \text{.}
\end{eqnarray*}%
Then the {\it Casorati curvature} of the horizontal space is defined as 
\[
C^{{\cal H}}=\frac{1}{s}\left\Vert B^{{\cal H}}\right\Vert ^{2}=\frac{1}{s}%
\sum_{\alpha =s+1}^{n_{2}}\sum_{i,j=1}^{s}\left( B_{ij}^{{\cal H}^{\alpha
}}\right) ^{2}. 
\]%
Let $L^{{\cal H}}$ be a $k$ dimensional subspace ($k\geq 2$) of the
horizontal space with orthonormal basis $\{h_{1},\dots ,h_{k}\}$. Then its
Casorati curvature $C^{L^{{\cal H}}}$ is given by 
\[
C^{L^{{\cal H}}}=\frac{1}{k}\sum_{\alpha =s+1}^{n_{2}}\sum_{i,j=1}^{k}\left(
B_{ij}^{{\cal H}^{\alpha }}\right) ^{2}. 
\]%
Moreover, the normalized $\delta ^{{\cal H}}${\it -Casorati curvatures} $%
\delta _{C}^{{\cal H}}(s-1)$ and $\hat{\delta}_{C}^{{\cal H}}(s-1)$
associated with the horizontal space at a point $p$ are given by 
\begin{equation}
\left[ \delta _{C}^{{\cal H}}(s-1)\right] _{p}=\frac{1}{2}C_{p}^{{\cal H}}+%
\frac{\left( s+1\right) }{2s}\inf \{C^{L^{{\cal H}}}|L^{{\cal H}}\ {\rm a\
hyperplane\ of\ }(\ker \pi _{\ast {p}})^{\perp }\},  \label{CC_H_RM_1}
\end{equation}%
and 
\begin{equation}
\left[ \hat{\delta}_{C}^{{\cal H}}(s-1)\right] _{p}=2C_{p}^{{\cal H}}-\frac{%
\left( 2s-1\right) }{2s}\sup \{C^{L^{{\cal H}}}|L^{{\cal H}}\ {\rm a\
hyperplane\ of\ }(\ker \pi _{\ast {p}})^{\perp }\}.  \label{CC_H_RM_2}
\end{equation}

\begin{lemma}
\label{General_Inq_Thm_RM} {\rm \cite[Singh et al. 2026]{SMM_2026}} Let $F:({%
M_{1}}^{m_{1}},g_{1})\rightarrow \left( {M_{2}}^{m_{2}},g_{2}\right) $ be a
Riemannian map between Riemannian manifolds with rank $r\geq 3$. Then 
\begin{equation}
\rho ^{{\cal H}}\leq \delta _{C}^{{\cal H}}(s-1)+\rho ^{{\cal R}}\quad {\rm %
and}\quad \rho ^{{\cal H}}\leq \hat{\delta}_{C}^{{\cal H}}(s-1)+\rho ^{{\cal %
R}}.  \label{eq-GIRMCASORATI}
\end{equation}%
The equality holds in any of the above two inequalities at a point $p\in {%
M_{1}}$ if and only if, for suitable orthonormal bases, the following hold. 
\[
B_{11}^{{\cal H}^{\alpha }}=B_{22}^{{\cal H}^{\alpha }}=\cdots
=B_{s-1\,s-1}^{{\cal H}^{\alpha }}=\frac{1}{2}B_{ss}^{{\cal H}^{\alpha }},
\]%
\[
B_{ij}^{{\cal H}^{\alpha }}=0,\quad 1\leq i\neq j\leq s.
\]%
The equality conditions can be interpreted as follows. The first condition
gives \\
$g_{2}(F_{\ast }h_{1},{\cal S}_{V_{\alpha }}F_{\ast
}h_{1})=g_{2}(F_{\ast }h_{2},{\cal S}_{V_{\alpha }}F_{\ast }h_{2})=\cdots
=g_{2}(F_{\ast }h_{s-1},{\cal S}_{V_{\alpha }}F_{\ast }h_{s-1})=\frac{1}{2}%
g_{2}(F_{\ast }h_{s},{\cal S}_{V_{\alpha }}F_{\ast }h_{s})$ with respect to
all directions $(V_{\alpha },\text{where}~\alpha \in \{s+1,\dots ,m_{2}\})$.
Equivalently, there exist $(m_{2}-s)$ mutually orthogonal unit vector fields
in $\left( {\rm range~}F_{\ast }\right) ^{\perp }$ such that shape operators
with respect to all directions have an eigenvalue of multiplicity $(s-1)$
and that for each $V_{\alpha }$ the distinguished eigendirections are the
same (namely $F_{\ast }h_{s}$). Hence, the leaves of range spaces are
invariantly quasi-umbilical {\rm \cite{DHV_2008}}. The second condition
gives $g_{2}(F_{\ast }h_{j},{\cal S}_{V_{\alpha }}F_{\ast }h_{i})=0$ with
respect to all directions $(V_{\alpha },\text{where}~\alpha \in \{s+1,\dots
,m_{2}\}$ in $\left( {\rm range~}F_{\ast }\right) ^{\perp })$. Equivalently,
the shape operator matrices become diagonal, and hence commute.
\end{lemma}

\begin{theorem}
Let $\pi :({N}_{1}^{n_{1}},g_{1})\rightarrow \left( N_{2}%
^{n_{2}},g_{2}\right) $ be a Riemannian map from an $n_{1}$-dimensional
Riemannian manifold {$N$}${_{1}}$ to an $n_{2}$-dimensional quaternionic
space form with rank $s\geq 3$. Then 
\begin{equation}
\rho ^{{\cal H}}\leq \delta _{C}^{{\cal H}}(s-1)+\frac{c}{4}+\frac{3c}{%
4s\left( s-1\right) }\sum_{\alpha =1}^{3}\left\Vert P_{\alpha }^{{\cal R}%
}\right\Vert ^{2}\quad {\rm and}\quad \rho ^{{\cal H}}\leq \hat{\delta}_{C}^{%
{\cal H}}(s-1)+\frac{c}{4}+\frac{3c}{4s\left( s-1\right) }\sum_{\alpha
=1}^{3}\left\Vert P_{\alpha }^{{\cal R}}\right\Vert ^{2}.
\label{Cas_qut_Inq_Eq_RM}
\end{equation}%
The equality case follows Lemma \ref{General_Inq_Thm_RM}.
\end{theorem}

\begin{proof}
At a point $p\in ${$N$}${_{1}}$, let $\left\{
h_{1},\dots ,h_{s}\right\} $, $\{\pi _{\ast }h_{1},\dots ,\pi _{\ast
}h_{s}\} $ and $\{V_{s+1},\dots ,V_{n_{2}}\}$ be orthonormal bases for $%
(\ker \pi _{\ast p})^{\perp }$, $({\rm range}~\pi _{\ast \pi(p)})$ and $({\rm range}%
~\pi _{\ast \pi(p)})^{\perp }$, respectively. Then by putting $W_{1}=W_{4}=h_{i}$
and $W_{2}=W_{3}=h_{j}$ in (\ref{eq-QSF}), (\ref{hor_range_isometry}) and (%
\ref{decompose_GCSF_RM}), we obtain 
\begin{eqnarray*}
&&\sum\limits_{i,j=1}^{s}R^{N_{2}}(F_{\ast }h_{i},F_{\ast }h_{j},F_{\ast
}h_{j},F_{\ast }h_{i}) \\
&=&\sum\limits_{i,j=1}^{s}\frac{c}{4}%
\{g_{1}(h_{j},h_{j})g_{1}(h_{i},h_{i})-g_{1}(h_{i},h_{j})g_{1}(h_{j},h_{i})\}
\\
&&+\sum_{\alpha =1}^{3}\sum\limits_{i,j=1}^{s}\frac{c}{4}\left\{ g_{2}\left(
\pi _{\ast }h_{i},P_{\alpha }^{{\cal R}}\pi _{\ast }h_{j}\right) g_{2}\left(
P_{\alpha }^{{\cal R}}\pi _{\ast }h_{j},\pi _{\ast }h_{i}\right)
-g_{2}\left( \pi _{\ast }h_{j},P_{\alpha }^{{\cal R}}\pi _{\ast
}h_{j}\right) g_{2}\left( P_{\alpha }^{{\cal R}}\pi _{\ast }h_{i},\pi _{\ast
}h_{i}\right) \right. \\
&&\left. +2g_{2}\left( \pi _{\ast }h_{i},P_{\alpha }^{{\cal R}}\pi _{\ast
}h_{j}\right) g_{2}\left( P_{\alpha }^{{\cal R}}\pi _{\ast }h_{j},\pi _{\ast
}h_{i}\right) \right\}
\end{eqnarray*}%
From (\ref{eq-scalH}) and (\ref{eq-PRalpha}), we obtain%
\begin{equation}
2{\tau }_{N_{2}}^{{\cal R}}=\frac{c}{4}s\left( s-1\right) +\frac{3c}{4}%
\sum_{\alpha =1}^{3}\left\Vert P_{\alpha }^{{\cal R}}\right\Vert ^{2}
\label{eq-GCI-(5)}
\end{equation}%
By using (\ref{eq-NSCALHR}),we obtain%
\begin{equation}
\rho ^{{\cal R}}=\frac{c}{4}+\frac{3c}{4s\left( s-1\right) }\sum_{\alpha
=1}^{3}\left\Vert P_{\alpha }^{{\cal R}}\right\Vert ^{2}.  \label{eq-GCI-(6)}
\end{equation}%
In view of (\ref{eq-GCI-(6)}) and (\ref{eq-GIRMCASORATI}), we get (\ref%
{Cas_qut_Inq_Eq_RM}).
\end{proof} 

\section{Riemannian submersion} \label{sec4}

Let $\left( N_{1},g_{1}\right) $ and $\left( N_{2},g_{2}\right) $ be two
Riemannian manifolds of dimension $n_{1}$ and $n_{2}$, respectively. A
surjective smooth map $F:\left( N_{1}^{n_{1}},g_{1}\right) \rightarrow
\left( N_{2}^{n_{2}},g_{2}\right) $ is called a {\it Riemannian submersion} if
its differential map (that is, $F_{\ast p}:T_{p}N_{1}\rightarrow
T_{F(p)}N_{2}$) is surjective at all $p\in N_{1}$ and $F_{\ast p}$ preserves
the length of all horizontal vectors at $p$.

Here, the tangent vectors to the fibers (fiber means $\{F^{-1}(q)\;:\ q\in
N_{2}\}$) is a submanifold of $N_{1}$ with $\dim (F^{-1}(q))=r=n_{1}-n_{2}$, are
vertical, while the orthogonal vectors to the fibers are horizontal.

Thus $TN_{1}$, the decomposition is a direct sum of two distributions: the
vertical distribution ${\cal V}=\ker F_{\ast }$ and its orthogonal
complement (known as the horizontal distribution) ${\cal H}=\left( \ker
F_{\ast }\right) ^{\perp }$.

Similarly, for each $p\in N_{1}$ the vertical and horizontal spaces in $%
T_{p}N_{1}$ are denoted by ${\cal V}_{p}=\left( \ker F_{\ast }\right) _{p}$
and ${\cal H}_{p}=\left( \ker F_{\ast }\right) _{p}^{\perp }$, respectively.

\subsubsection*{O'Neill tensors}

The geometry of Riemannian submersions is characterised by O'Neill's tensors 
{${\cal T}$} and {${\cal A}$} defined for vector fields $E$, $F$ on $N_{1}$
by 
\begin{eqnarray*}
{\cal T}\left( E,F\right) &=&{{\cal T}_{E}F=h\nabla _{vE}vF+v\nabla _{vE}hF}
\\
{\cal A}\left( E,F\right) &=&{\cal A}_{E}F=v{\nabla _{hE}hF+h\nabla _{hE}vF,}
\end{eqnarray*}%
where $\nabla $ is the Levi-Civita connection on $g_{1}$, $h$ and $v$ are
projection morphisms of $E$, $F\in TN_{1}$ to $\left( \ker F_{\ast }\right)
^{\perp }$ and $\ker F_{\ast }$, respectively.

The O'Neill tensors also satisfy: 
\[
{\cal A}_{Y_{1}}Y_{2}=-{\cal A}_{Y_{2}}Y_{1},\quad {\cal T}_{U_{1}}U_{2}=%
{\cal T}_{U_{2}}U_{1}, 
\]%
and 
\[
g_{1}\left( {\cal T}_{E}F,G\right) =-g_{1}\left( F,{\cal T}_{E}G\right)
,\quad g_{1}\left( {\cal A}_{E}F,G\right) =-g_{1}\left( F,{\cal A}%
_{E}G\right) , 
\]%
where $U_{1},U_{2}\in \ker F_{\ast }$, $Y_{1},Y_{2},Y_{3}\in \left( \ker
F_{\ast }\right) ^{\perp }$ and $E$, \thinspace $F$, $G\in TN_{1}$ \cite%
{Neill_1966}.

\subsubsection*{Relations between Riemannian curvature tensors}

Let $R^{N_{1}}$, $R^{N_{2}}$, $R^{\ker F_{\ast }}$, and $R^{\left( \ker
F_{\ast }\right) ^{\perp }}$ denote the Riemannian curvature tensors
corresponding to $N_{1}$, $N_{2}$, $\ker F_{\ast }$, and $\left( \ker
F_{\ast }\right) ^{\perp }$, respectively. Then, we have 
\begin{eqnarray}
R^{N_{1}}\left( F_{1},F_{2},F_{3},F_{4}\right)
&=&R^{\ker F_{\ast}}\left( F_{1},F_{2},F_{3},F_{4}\right)
-g_{1}(T_{F_{1}}F_{4},T_{F_{2}}F_{3})  \nonumber \\
&&+g_{1}(T_{F_{2}}F_{4},T_{F_{1}}F_{4}),  \label{Gauss_Codazz_RS_T}
\end{eqnarray}%
\begin{eqnarray}
R^{N_{1}}\left( X_{1},X_{2},X_{3},X_{4}\right) &=&R^{\left( \ker F_{\ast
}\right) ^{\perp }}\left( X_{1},X_{2},X_{3},X_{4}\right) +2g_{1}\left(
A_{X_{1}}X_{2},A_{X_{3}}X_{4}\right)  \nonumber \\
&&-g_{1}\left( A_{X_{2}}X_{3},A_{X_{1}}X_{4}\right) +g_{1}\left(
A_{X_{1}}X_{3},A_{X_{2}}X_{4}\right)  \label{Gauss_Codazz_RS_A}
\end{eqnarray}%
\begin{eqnarray}
R^{N_{1}}\left( X_{1},F_{1},X_{2},F_{2}\right) &=&g_{1}\left( \left( \nabla
_{X_{1}}^{1}{\cal T}\right) \left( F_{1},F_{2}\right) ,X_{2}\right)
+g_{1}\left( \left( \nabla _{F_{1}}^{1}{\cal A}\right) \left(
X_{1},X_{2}\right) ,F_{2}\right)  \nonumber \\
&&-g_{1}\left( {\cal T}_{F_{1}}X_{1},{\cal T}_{F_{2}}X_{2}\right)
+g_{1}\left( {\cal A}_{X_{2}}F_{2},{\cal A}_{X_{1}}F_{1}\right) ,
\label{eq-P-(12)}
\end{eqnarray}%
for all $X_{1},X_{2},X_{3},X_{4}\in \left( \ker F_{\ast }\right) ^{\perp }$\
and $F_{1},F_{2},F_{3},F_{4}\in \ker F_{\ast }$. Here, $\nabla ^{1}$ is the
Levi-Civita connection with respect to the metric $g_{1}$ \cite%
{Falcitelli_2004,Neill_1966}.

\subsection{Casorati inequalities from quaternionic space form for vertical
distributions of Riemannian submersions} \label{sec5}

Let $\left\{ v_{1},\ldots ,v_{\ell }\right\} $ and $\left\{ h_{1},\ldots
,h_{s}\right\} $ be orthonormal bases of $\left( \ker F_{\ast p}\right) $
and $\left( \ker F_{\ast p}\right) ^{\perp }$ at point $p\in N_{1}$. Then
the scalar curvature defined by 
\begin{equation}
2{\tau }_{{\cal V}}^{\ker F_{\ast }}=\sum\limits_{i,j=1}^{\ell }R^{\ker F_{\ast}}\left(v_{i},v_{j},v_{j},v_{i}\right) ,\quad 2{\tau }_{{\cal V}%
}^{N_{1}}=\sum\limits_{i,j=1}^{\ell }R^{N_{1}}\left(
v_{i},v_{j},v_{j},v_{i}\right) .  \label{eq-scalV}
\end{equation}%
Consequently, we define the normalized scalar curvature as 
\begin{equation}
\rho _{{\cal V}}^{\ker F_{\ast }}=\frac{2{\tau }_{{\cal V}}^{\ker F_{\ast }}%
}{\ell \left( \ell -1\right) },\quad \rho _{{\cal V}}^{N_{1}}=\frac{2{\tau }%
_{{\cal V}}^{N_{1}}}{\ell \left( \ell -1\right) }  \label{eq-Nscal}
\end{equation}%
Also, we set 
\begin{eqnarray*}
T_{ij}^{{\cal H}^{\alpha }} &=&g_{N_{1}}\left( T_{v_{i}}v_{j},h_{\alpha
}\right) ,\quad i,j=1,\dots ,\ell ,\quad \alpha =1,\dots ,s, \\
\left\Vert T^{{\cal H}}\right\Vert ^{2} &=&\sum_{i,j=1}^{\ell }g_{1}\left(
T_{v_{i}}v_{j},T_{v_{i}}v_{j}\right) , \\
{\rm trace~}T^{{\cal H}} &=&\sum_{i=1}^{\ell }T_{v_{i}}v_{i}, \\
\left\Vert {\rm trace~\,}T^{{\cal H}}\right\Vert ^{2} &=&g_{1}\left( {\rm %
trace}~T^{{\cal H}},{\rm trace}~T^{{\cal H}}\right) .
\end{eqnarray*}%
Then the {\it Casorati curvature} of the vertical space is defined as 
\[
C^{{\cal V}}=\frac{1}{\ell }\left\Vert T^{{\cal H}}\right\Vert ^{2}=\frac{1}{%
\ell }\sum_{\alpha =1}^{s}\sum_{i,j=1}^{\ell }\left( T_{ij}^{{\cal H}%
^{\alpha }}\right) ^{2}. 
\]

Let $L^{{\cal V}}$ be a $k$ dimensional subspace ($k\geq 2$) of vertical
space with orthonormal basis $\{v_{1},\dots ,v_{k}\}$. Then its Casorati
curvature $C^{L^{{\cal V}}}$ is given by 
\[
C^{L^{{\cal V}}}=\frac{1}{k}\sum_{\alpha =1}^{s}\sum_{i,j=1}^{k}\left(
T_{ij}^{{\cal H}^{\alpha }}\right) ^{2}. 
\]%
Moreover, the normalized $\delta ^{{\cal V}}${\it -Casorati curvatures} $%
\delta _{C}^{{\cal V}}(\ell -1)$ and $\hat{\delta}_{C}^{{\cal V}}(\ell -1)$
associated with the vertical space at a point $p$ are given by 
\[
\left[ \delta _{C}^{{\cal V}}(\ell -1)\right] _{p}=\frac{1}{2}C_{p}^{{\cal V}%
}+\frac{\left( \ell +1\right) }{2\ell }\inf \{C^{L^{{\cal V}}}|L^{{\cal V}}\ 
{\rm a\ hyperplane\ of\ }(\ker F_{\ast {p}})\}, 
\]%
and 
\[
\left[ \hat{\delta}_{C}^{{\cal V}}(\ell -1)\right] _{p}=2C_{p}^{{\cal V}}-%
\frac{\left( 2\ell -1\right) }{2\ell }\sup \{C^{L^{{\cal V}}}|L^{{\cal V}}\ 
{\rm a\ hyperplane\ of\ }(\ker F_{\ast {p}})\}. 
\]

\begin{lemma}
\label{General_Inq_Thm_VD_RS} {\rm \cite[Singh et al. 2026]{SMM_2026}} Let $\pi :({N}_{1},g_{1})\rightarrow \left( {N}%
_{2},g_{2}\right) $ be a Riemannian submersion between Riemannian manifolds
with vertical space of dimension $\ell \geq 3$, where $\dim N_{1}=n_{1}$ and $%
\dim N_{2}=n_{2}$. Then 
\begin{equation}
\rho _{{\cal V}}^{\ker F_{\ast }}\leq \delta _{C}^{{\cal V}}\left(
\ell -1\right) +{\rho }_{{\cal V}}^{N_{1}}\quad {\rm and}\quad \rho _{{\cal V}%
}^{\ker F_{\ast }}\leq \hat{\delta}_{C}^{{\cal V}}\left( \ell -1\right) +{\rho }%
_{{\cal V}}^{N_{1}}.  \label{eq-GCIV-(1a)}
\end{equation}%
Moreover, the equality holds in any of the above two inequalities at a point 
$p\in {N_{1}}$ if and only if for suitable orthonormal bases, the following
hold 
\[
T_{11}^{{\cal H}^{\alpha }}=T_{22}^{{\cal H}^{\alpha }}=\cdots =T_{\ell
-1\,\ell -1}^{{\cal H}^{\alpha }}=\frac{1}{2}T_{\ell \ell }^{{\cal H}%
^{\alpha }}, 
\]%
\[
T_{ij}^{{\cal H}^{\alpha }}=0,\quad 1\leq i\neq j\leq \ell . 
\]%
The equality conditions can be interpreted as follows. The first condition
gives $g_{1}(v_{1},T_{v_{1}}h_{\alpha })=g_{1}(v_{2},T_{v_{2}}h_{\alpha
})=\cdots =g_{1}(v_{\ell -1},T_{v_{\ell -1}}h_{\alpha })=\frac{1}{2}%
g_{1}(v_{\ell},T_{v_{\ell}}h_{\alpha })$ with respect to all horizontal directions 
$(h_{\alpha },\text{where}~\alpha \in \{1,\dots ,s\})$. Equivalently, there
exist $(m_{1}-\ell )$ mutually orthogonal horizontal unit vector fields such
that the shape operator with respect to all directions has an eigenvalue of
multiplicity $(\ell -1)$ and that for each $h_{\alpha }$ the distinguished
eigendirections are the same (namely $v_{\ell}$). Hence, the leaves of vertical
space (called fibers of $F$) are invariantly quasi-umbilical \cite{DHV_2008}%
. The second condition gives $g_{1}(v_{j},T_{v_{i}}h_{\alpha })=0$ with
respect to all horizontal directions $(h_{\alpha },\text{where}~\alpha \in
\{1,\dots ,s\})$. Equivalently, the shape operator matrices become diagonal,
and hence commute.
\end{lemma}

\begin{theorem}
Let $\pi:(${$N$}${_{1}},J_{\alpha },g_{1})\rightarrow \left( {N_{2}}%
,g_{2}\right) $ be a Riemannian submersion from $n_{1}$-dimensional
quaternionic space form onto a $n_{2}$-dimensional Riemannian manifold with
vertical space of dimension $\ell \geq 3$. Then 
\begin{equation}
\rho _{{\cal V}}^{\ker F_{\ast }}\leq \delta _{C}^{{\cal V}}\left( \ell
-1\right) +\frac{c}{4}+\frac{3c}{4\ell \left( \ell -1\right) }\sum_{\alpha
=1}^{3}\Vert Q_{\alpha }\Vert ^{2}, \ \rho _{{\cal V}%
}^{\ker F_{\ast }}\leq \hat{\delta}_{C}^{{\cal V}}\left( \ell -1\right) +%
\frac{c}{4}+\frac{3c}{4\ell \left( \ell -1\right) }\sum_{\alpha =1}^{3}\Vert
Q_{\alpha }\Vert ^{2}.  \label{eq-Cas_quat_CIV-(1a)}
\end{equation}%
The equality case follows Lemma \ref{General_Inq_Thm_VD_RS}.
\end{theorem}

\begin{proof}
At a point $p\in ${$N$}${_{1}}$, let $\left\{
v_{1},\dots ,v_{\ell }\right\} $ and $\{h_{1},\dots ,h_{s}\}$ be orthonormal
bases for vertical and horizontal space, respectively. Then putting $%
F_{1}=F_{4}=v_{i}$ and $F_{2}=F_{3}=v_{j}$ in (\ref{Gauss_Codazz_RS_T}),
using (\ref{eq-QSF}) and (\ref{decompose_GCSF_VRS}), we obtain%
\begin{eqnarray*}
&&\sum\limits_{i,j=1}^{\ell }R^{N_{1}}(v_{i},v_{j},v_{j},v_{i}) \\
&=&\sum\limits_{i,j=1}^{\ell }\frac{c}{4}%
\{g_{1}(v_{j},v_{j})g_{1}(v_{i},v_{i})-g_{1}(v_{i},v_{j})g_{1}(v_{j},v_{i})\}
\\
&&+\sum_{\alpha =1}^{3}\sum\limits_{i,j=1}^{\ell }\frac{c}{4}\left\{
g_{1}\left( v_{i},Q_{\alpha }v_{j}\right) g_{1}\left( Q_{\alpha
}v_{j},v_{i}\right) -g_{1}\left( v_{j},Q_{\alpha }v_{j}\right) g_{1}\left( Q_{\alpha }v_{i},v_{i}\right) \right. \\
&&\left. +2g_{1}\left( v_{i},Q_{\alpha }v_{j}\right) g_{1}\left(
P_{\alpha }v_{j},v_{i}\right) \right\}
\end{eqnarray*}%
From (\ref{eq-scalV}), we obtain%
\begin{equation}
2{\tau }_{{\cal V}}^{N_{1}}=\ell \left( \ell -1\right) \frac{c}{4}+\frac{3c}{%
4}\sum_{\alpha =1}^{3}\Vert {Q_{\alpha }}\Vert ^{2}
\label{eq-GCIV-(4)}
\end{equation}%
By using (\ref{eq-Nscal}) and (\ref{eq-GCIV-(4)}), we get%
\begin{equation}
{\rho }_{{\cal V}}^{N_{1}}=\frac{c}{4}+\frac{3c}{4\ell \left( \ell -1\right) 
}\sum_{\alpha =1}^{3}\Vert {Q_{\alpha }}\Vert ^{2}.
\label{eq-GCIV-(5)}
\end{equation}%
In view of (\ref{eq-GCIV-(5)}) and (\ref{eq-GCIV-(1a)}), we get (\ref%
{eq-Cas_quat_CIV-(1a)}).
\end{proof}

\subsection{Casorati inequalities for horizontal distributions of Riemannian
submersions from quaternionic space form} \label{sec6}

Let $\left\{ h_{1},\ldots ,h_{s}\right\} $ and $\left\{ v_{1},\ldots
,v_{\ell }\right\} $ be orthonormal bases of $\left( \ker F_{\ast p}\right)
^{\perp }$ and $\ker F_{\ast p}$, respectively, at a point $p\in N_{1}$. Then
we define scalar curvature

\begin{equation}
2\tau _{{\cal H}}^{({\ker F_{\ast}})^{\perp}}=\sum\limits_{i,j=1}^{s}R^{(\ker F_{\ast })^{\perp}}\left( h_{i},h_{j},h_{j},h_{i}\right) ,\ 2\tau _{{\cal H}%
}^{N_{1}}=\sum\limits_{i,j=1}^{s}R^{N_{1}}\left(
h_{i},h_{j},h_{j},h_{i}\right)  \label{eq-Hscal}
\end{equation}%
Consequently, we define the normalized scalar curvatures as 
\begin{equation}
{\rho }^{{\cal H}}=\frac{2\tau _{{\cal H}}^{({\ker F_{\ast}})^{\perp}}}{s\left(
s-1\right) },\quad \rho _{{\cal H}}^{N_{1}}=\frac{2\tau _{{\cal H}}^{N_{1}}}{%
s\left( s-1\right) }  \label{eq-Nhscal}
\end{equation}%
Also, we set 
\begin{eqnarray*}
A_{ij}^{{\cal V}^{\alpha }} &=&g_{1}\left( A_{h_{i}}h_{j},v_{\alpha }\right)
,\quad i,j=1,\dots ,s,\quad \alpha =1,\dots ,\ell , \\
\left\Vert A^{{\cal V}}\right\Vert ^{2} &=&\sum_{i,j=1}^{s}g_{1}\left(
A_{h_{i}}h_{j},A_{h_{i}}h_{j}\right) , \\
{\rm trace~}A^{{\cal V}} &=&\sum_{i=1}^{s}A_{h_{i}}h_{i}, \\
\left\Vert {\rm trace~\,}A^{{\cal V}}\right\Vert ^{2} &=&g_{1}\left( {\rm %
trace}~A^{{\cal V}},{\rm trace}~A^{{\cal V}}\right) .
\end{eqnarray*}%
Then the {\it Casorati curvature} of the horizontal space is defined as 
\[
C^{{\cal H}}=\frac{1}{s}\left\Vert A^{{\cal V}}\right\Vert ^{2}=\frac{1}{s}%
\sum_{\alpha =1}^{\ell }\sum_{i,j=1}^{s}\left( A_{ij}^{{\cal V}^{\alpha
}}\right) ^{2}. 
\]%
Let $L^{{\cal H}}$ be a $k$ dimensional subspace ($k\geq 2$) of horizontal
space with orthonormal basis $\{h_{1},\dots ,h_{k}\}$. Then its Casorati
curvature $C^{L^{{\cal H}}}$ is given by 
\[
C^{L^{{\cal H}}}=\frac{1}{k}\sum_{\alpha =1}^{\ell }\sum_{i,j=1}^{k}\left(
A_{ij}^{{\cal H}^{\alpha }}\right) ^{2}. 
\]%
Moreover, the normalized $\delta ${\it $^{{\cal H}}$-Casorati curvatures} $%
\delta _{C}^{{\cal H}}(s-1)$ and $\hat{\delta}_{C}^{{\cal H}}(s-1)$
associated with the horizontal space at a point $p$ are given by (\ref%
{CC_H_RM_1}) and (\ref{CC_H_RM_2}).

\begin{lemma}
\label{General_Inq_Thm_HD_RS} {\rm \cite[Singh et al. 2026]{SMM_2026}} Let $F:(M_{1}^{m_{1}},g_{1})\rightarrow \left(
M_{2}^{m_{2}},g_{2}\right) $ be a Riemannian submersion between Riemannian
manifolds with horizontal space of dimension $r\geq 3$. Then 
\begin{equation}
\rho ^{{\cal H}}\leq \delta _{C}^{{\cal H}}\left( r-1\right)
+\rho ^{N_{1}}_{{\cal H}}\quad {\rm and}\quad \rho ^{{\cal H}}\leq \hat{%
\delta}_{C}^{{\cal H}}\left( r-1\right) +\rho ^{N_{1}}_{{\cal H}}.
\label{eq-GCIH-(1a)}
\end{equation}%
Moreover, equality holds in any of the above two inequalities at a point $%
p\in M_{1}$ if and only if the tensor $A$ vanishes. Geometrically, equality
means that the horizontal distribution is integrable.    
\end{lemma}

\begin{theorem}
Let $\pi:(N_{1}^{n_{1}},J_{\alpha },g_{1})\rightarrow \left( N_{2}^{n_{2}},g_{2}\right) $ be a Riemannian submersion from quaternionic space form onto a Riemannian
manifolds, where $\dim N_{1}=n_{1}$ and $\dim N_{2}=n_{2}$ with $\dim \left(
\ker F_{\ast }\right) ^{\perp }=s\geq 3$. Then 
\begin{equation}
\rho ^{{\cal H}}\leq \delta _{C}^{{\cal H}}\left( s-1\right) +%
\frac{c}{4}+\frac{3c}{4s\left( s-1\right) }\sum_{\alpha =1}^{3}\Vert {%
P_{\alpha }}\Vert ^{2}\quad {\rm and}\quad \rho ^{{\cal H}}\leq 
\hat{\delta}_{C}^{{\cal H}}\left( s-1\right) +\frac{c}{4}+\frac{3c}{4s\left(
s-1\right) }\sum_{\alpha =1}^{3}\Vert {P_{\alpha }}\Vert ^{2}
\label{eq-Cas_quat_H-(1a)}
\end{equation}%
The equality case follows Lemma \ref{General_Inq_Thm_HD_RS}.
\end{theorem}

\begin{proof}
Let $\left\{ h_{1},\ldots ,h_{s}\right\} $ and $%
\left\{ v_{1},\ldots ,v_{\ell }\right\} $ be orthonormal bases of $\left(
\ker F_{\ast p}\right) ^{\perp }$ and $\ker F_{\ast }$, respectively, at a
point $p\in N_{1}$. Then putting $X_{1}=X_{4}=h_{i}$ and $X_{2}=X_{3}=h_{j}$
in (\ref{eq-QSF}) and (\ref{decompose_GCSF_VRS}), we obtain%
\begin{eqnarray*}
&&\sum\limits_{i,j=1}^{s}R^{N_{1}}(h_{i},h_{j},h_{j},h_{i})
\\
&=&\sum\limits_{i,j=1}^{s}\frac{c}{4}%
\{g_{1}(h_{j},h_{j})g_{1}(h_{i},h_{i})-g_{1}(h_{i},h_{j})g_{1}(h_{j},h_{i})\}
\\
&&+\sum_{\alpha =1}^{3}\sum\limits_{i,j=1}^{s}\frac{c}{4}\left\{ g_{1}\left(
h_{i},P_{\alpha }h_{j}\right) g_{1}\left( P_{\alpha }h_{j},h_{i}\right) -g_{1}\left( h_{j},P_{\alpha }h_{j}\right)
g_{1}\left( P_{\alpha }h_{i},h_{i}\right) \right. \\
&&\left. +2g_{1}\left( h_{i},P_{\alpha }h_{j}\right) g_{1}\left(
P_{\alpha }h_{j},h_{i}\right) \right\}
\end{eqnarray*}%
From (\ref{eq-Hscal}), we obtain%
\begin{equation}
2\tau _{{\cal H}}^{N_{1}}=s\left( s-1\right) \frac{c}{4}+3\frac{c}{4}%
\sum_{\alpha =1}^{3}\Vert {P_{\alpha }}\Vert ^{2}  \label{eq-GCIH-(2)}
\end{equation}%
From (\ref{eq-Nhscal}) and (\ref{eq-GCIH-(2)}), we get%
\begin{equation}
\rho _{{\cal H}}^{N_{1}}=\frac{c}{4}+3\frac{c}{4s\left( s-1\right) }%
\sum_{\alpha =1}^{3}\Vert {P_{\alpha }}\Vert ^{2}.  \label{eq-GCIH-(3)}
\end{equation}%
In view of (\ref{eq-GCIH-(1a)}) and (\ref{eq-GCIH-(3)}), we get (\ref%
{eq-Cas_quat_H-(1a)}).
\end{proof}

\subsection{Casorati inequalities for horizontal and vertical distributions of Riemannian submersions from quaternionic space form} \label{sec7}

\begin{lemma}
\label{Th Cas VH RS} {\rm \cite[Singh 2026]{Singh_2026}}Let $\pi:(N_{1}^{n_{1}},J_{\alpha },g_{1})\rightarrow \left(
N_{2}^{n_{2}},g_{2}\right) $ be a Riemannian submersion between two Riemannian
manifolds, where $\dim N_{1}=n_{1}$ and $\dim N_{2}=n_{2}$ with $\dim \left(
\ker F_{\ast }\right) ^{\perp }=s\geq 3$, and $\dim \left( \ker F_{\ast
}\right) =\ell \geq 3$. Then, we have 
\begin{eqnarray}
\frac{\rho ^{{\cal H}}}{\ell (\ell -1)}+\frac{\rho ^{{\cal V}}}{s(s-1)}
&\leq &\frac{1}{s(s-1)}\delta _{C}^{{\cal V}}\left( \ell -1\right) +\frac{1}{%
\ell (\ell -1)}\delta _{C}^{{\cal H}}\left( s-1\right) +\frac{\rho _{{\cal V}%
}^{N_{1}}}{s(s-1)}+\frac{\rho _{{\cal H}}^{N_{1}}}{\ell (\ell -1)}  \nonumber
\\
&&+\frac{2}{s(s-1)\ell (\ell -1)}\sum_{i=1}^{s}\sum_{j=1}^{\ell}R^{M_{1}}\left(
h_{i},v_{j},v_{j},h_{i}\right)  \nonumber \\
&&+\frac{1}{s(s-1)\ell (\ell -1)}\left( 2\delta \left( N\right) -\left\Vert
T^{{\cal V}}\right\Vert ^{2}+\left\Vert A^{{\cal H}}\right\Vert ^{2}\right)
\label{eq-CAS HV GEN 1}
\end{eqnarray}%
and%
\begin{eqnarray}
\frac{\rho ^{{\cal H}}}{\ell (\ell -1)}+\frac{\rho ^{{\cal V}}}{s(s-1)}
&\leq &\frac{1}{s(s-1)}\hat{\delta}_{C}^{{\cal V}}\left( \ell -1\right) +%
\frac{1}{\ell (\ell -1)}\hat{\delta}_{C}^{{\cal H}}\left( s-1\right) +\frac{%
\rho _{{\cal V}}^{N_{1}}}{s(s-1)}+\frac{\rho _{{\cal H}}^{N_{1}}}{\ell (\ell
-1)}  \nonumber \\
&&+\frac{2}{s(s-1)\ell (\ell -1)}\sum_{i=1}^{s}\sum_{j=1}^{\ell}R^{M_{1}}\left(
h_{i},v_{j},v_{j},h_{i}\right)  \nonumber \\
&&+\frac{1}{s(s-1)\ell (\ell -1)}\left( 2\delta \left( N\right) -\left\Vert
T^{{\cal V}}\right\Vert ^{2}+\left\Vert A^{{\cal H}}\right\Vert ^{2}\right)
\label{eq-CAS HV GEN 2}
\end{eqnarray}%
Moreover, the equality holds in any of the above two inequalities at a point 
$p\in ${$N$}${_{1}}$ if and only if 
\[
T_{11}^{{\cal H}^{\alpha }}=T_{22}^{{\cal H}^{\alpha }}=\cdots =T_{\ell
-1\,\ell -1}^{{\cal H}^{\alpha }}=\frac{1}{2}T_{\ell \ell }^{{\cal H}%
^{\alpha }}, 
\]%
\[
T_{ij}^{{\cal H}^{\alpha }}=0,\quad 1\leq i\neq j\leq \ell . 
\]%
\[
A_{ij}^{{\cal V}^{\alpha }}=0,\quad 1\leq i,j\leq s . 
\]
The equality conditions can be interpreted as follows. The first condition gives $g_1(v_1, T_{v_1} h_\alpha) = g_1(v_2, T_{v_2} h_\alpha) = \cdots =g_1(v_{\ell-1}, T_{v_{\ell-1}} h_\alpha) = \frac{1}{2}g_1(v_{\ell}, T_{v_{\ell}} h_\alpha)$ with respect to all horizontal directions $(h_\alpha, \text{where}~ \alpha \in \{1, \dots, s\})$. Equivalently, there exist $s$ mutually orthogonal horizontal unit vector fields such that the shape operator with respect to all directions has an eigenvalue of multiplicity $(\ell-1)$ and that for each $h_\alpha$ the distinguished eigendirections are the same (namely $v_\ell$). Hence, the leaves of vertical space (called fibers of $F$) are invariantly quasi-umbilical \cite{DHV_2008}. The second condition gives $g_1(v_j, T_{v_i} h_\alpha) = 0$ with respect to all horizontal directions $(h_\alpha, \text{where}~ \alpha \in \{1, \dots, s\})$. Equivalently, the shape operator matrices become diagonal and hence commute. The third condition shows that $A$ vanishes on horizontal space; geometrically, horizontal space is integrable.  
\end{lemma}

\begin{theorem}
Let $\pi:(N_{1}^{n_{1}},J_{\alpha },g_{1})\rightarrow \left( N_{2}^{n_{2}},g_{2}\right) $ be a Riemannian submersion from quaternionic space form onto a Riemannian
manifold, where $\dim N_{1}=n_{1}=4m$ and $\dim N_{2}=n_{2}$ with $\dim
\left( \ker F_{\ast }\right) ^{\perp }=s\geq 3$, and $\dim \left( \ker
F_{\ast }\right) =\ell \geq 3$. Then, we have 
\begin{eqnarray}
\frac{\rho ^{{\cal H}}}{\ell (\ell -1)}+\frac{\rho ^{{\cal V}}}{s(s-1)}
&\leq &\frac{1}{s(s-1)}\delta _{C}^{{\cal V}}\left( \ell -1\right) +\frac{1}{%
\ell (\ell -1)}\delta _{C}^{{\cal H}}\left( s-1\right)  \nonumber \\
&&+\frac{c}{4s(s-1)\ell (\ell -1)}\left( \ell ^{2}+s^{2}+2s\ell -\ell
-s\right)  \nonumber \\
&&+\frac{3c}{4s(s-1)\ell (\ell -1)}\sum_{\alpha =1}^{3}\left( \Vert
Q_{\alpha }\Vert ^{2}+\Vert {P_{\alpha }}\Vert ^{2}+2\Vert {P_{\alpha }^{%
{\cal V}}}\Vert ^{2}\right)  \nonumber \\
&&+\frac{1}{s(s-1)\ell (\ell -1)}\left( 2\delta \left( N\right) -\left\Vert
T^{{\cal V}}\right\Vert ^{2}+\left\Vert A^{{\cal H}}\right\Vert ^{2}\right)
\label{eq CAS QSF 1}
\end{eqnarray}%
and%
\begin{eqnarray}
\frac{\rho ^{{\cal H}}}{\ell (\ell -1)}+\frac{\rho ^{{\cal V}}}{s(s-1)}
&\leq &\frac{1}{s(s-1)}\hat{\delta}_{C}^{{\cal V}}\left( \ell -1\right) +%
\frac{1}{\ell (\ell -1)}\hat{\delta}_{C}^{{\cal H}}\left( s-1\right) 
\nonumber \\
&&+\frac{c}{4s(s-1)\ell (\ell -1)}\left( \ell ^{2}+s^{2}+2s\ell -\ell
-s\right)  \nonumber \\
&&+\frac{3c}{4s(s-1)\ell (\ell -1)}\sum_{\alpha =1}^{3}\left( \Vert
Q_{\alpha }\Vert ^{2}+\Vert {P_{\alpha }}\Vert ^{2}+2\Vert {P_{\alpha }^{%
{\cal V}}}\Vert ^{2}\right)  \nonumber \\
&&+\frac{1}{s(s-1)\ell (\ell -1)}\left( 2\delta \left( N\right) -\left\Vert
T^{{\cal V}}\right\Vert ^{2}+\left\Vert A^{{\cal H}}\right\Vert ^{2}\right)
\label{eq CAS QSF 2}
\end{eqnarray}%
The equality case follows Lemma \ref{Th Cas VH RS}.
\end{theorem}

\begin{proof}
Let $\left\{ h_{1},\ldots ,h_{s}\right\} $ and $%
\left\{ v_{1},\ldots ,v_{\ell }\right\} $ be orthonormal bases of $\left(
\ker F_{\ast p}\right) ^{\perp }$ and $\ker F_{\ast }$, respectively, at a
point $p\in N_{1}$, and $N_{1}$ is a quaternionic space form then we have,
By using (\ref{eq-QSF}), we obtain%
\begin{equation}
\sum_{i=1}^{s}\sum_{j=1}^{r}R^{M_{1}}\left( h_{i},V_{j},V_{j},h_{i}\right) =%
\frac{c}{4}s\ell +\frac{3c}{4}\sum_{\alpha =1}^{3}\Vert {P_{\alpha }^{{\cal V%
}}}\Vert ^{2}  \label{eq-scalHV}
\end{equation}%
In view of (\ref{eq-GCIV-(5)}), (\ref{eq-GCIH-(3)}), (\ref{eq-scalHV}), (\ref%
{eq-CAS HV GEN 1}) and (\ref{eq-CAS HV GEN 2}), we get (\ref{eq CAS QSF 1})
and (\ref{eq CAS QSF 2}).
\end{proof}


\begin{thebibliography}{99}
\bibitem{ALVY} Aquib, M., Lee, J.W., V\^ilcu, G.E., Yoon, D.W.:
Classification of Casorati ideal Lagrangian submanifolds in complex space
forms. Differential Geom. Appl., {\bf 63} (2019), 30-49.

\bibitem{ASJ} Aquib, M., Shahid, M.H., Jamali, M.: Lower extremities for
generalized normalized $\delta$-Casorati curvatures of bi-slant submanifolds
in generalized complex space forms. Kragujevac J. Math., {\bf 42}(4) (2018),
591-605.

\bibitem{Blair} Blair, D.E.: {Riemannian Geometry of Contact and Symplectic
Manifolds}. Birkh\"auser, Boston, (2010).

\bibitem{Casorati_1890} Casorati, F.: Mesure de la courbure des surfaces
suivant l'id\'ee commune. Acta Math., {\bf 14} (1890), 95-110.

\bibitem{Chen_Survey} Chen, B.Y.: Recent developments in $\delta$-Casorati
curvature invariants. Turkish J. Math., {\bf 45}(1) (2021), 1-46.

\bibitem{DHV_2008} Decu, S., Haesen, S., Verstraelen, L.: Optimal
inequalities characterising quasi-umbilical submanifolds. J. Inequal. Pure
Appl. Math., {\bf 9}(3) (2008), 1-7.

\bibitem{Falcitelli_2004} Falcitelli, M., Ianus, S., Pastore, A.M.: {%
Riemannian Submersions and Related Topics}. World Scientific, River Edge,
NJ, (2004).

\bibitem{Fischer_1992} Fischer, A.E.: {Riemannian maps between Riemannian
manifolds}. Contemp. Math., {\bf 132} (1992), 331-366.

\bibitem{GRK_book} Garc\'ia-R\'io, E., Kupeli, D.N.: {Semi-Riemannian Maps
and Their Applications}. Kluwer, (1999).

\bibitem{Ghisoiu} Ghi\c{s}oiu, V.: Inequalities for the Casorati curvatures
of slant submanifolds in complex space forms. Proc. Riemannian Geometry and
Applications, Bucharest, Romania, (2011), 145-150.


\bibitem{Ishihara_1974} Ishihara, S.: {Quaternion Kaehlerian manifolds}.
Journal of Differential Geometry, {\bf 9}(4), 483-500.


\bibitem{LLSV} Lee, C.W., Lee, J.W., \c{S}ahin, B., V\^ilcu, G.E.: Optimal
inequalities for Riemannian maps and Riemannian submersions involving
Casorati curvatures. Ann. Mat. Pura Appl., {\bf 200}(3) (2021), 1277-1295.

\bibitem{LLV_2020} Lee, J.W., Lee, C.W., V\^ilcu, G.E.: Classification of
Casorati ideal Legendrian submanifolds in Sasakian space forms. J. Geom.
Phys., {\bf 155} (2020), 1-13.

\bibitem{LLV_2022} Lee, C.W., Lee, J.W., V\^ilcu, G.E.: Classification of
Casorati ideal Legendrian submanifolds in Sasakian space forms II. J. Geom.
Phys., {\bf 171} (2022), 1-10.

\bibitem{LLV_2017} Lee, C.W., Lee, J.W., V\^ilcu, G.E.: Optimal inequalities
for the normalized $\delta$-Casorati curvatures of submanifolds in Kenmotsu
space forms. Adv. Geom., {\bf 17}(3) (2017), 355-362.


\bibitem{Lone_2017a} Lone, M.A.: An inequality for generalized normalized $%
\delta$-Casorati curvatures of slant submanifolds in generalized complex
space form. Balkan J. Geom. Appl., {\bf 22}(1) (2017), 41-50.

\bibitem{Lone_2017} Lone, M.A.: Some inequalities for generalized normalized 
$\delta$-Casorati curvatures of slant submanifolds in generalized Sasakian
space form. Novi Sad J. Math., {\bf 47}(1) (2017), 129-141.

\bibitem{Lone_2019} Lone, M.A.: Optimal inequalities for generalized
normalized $\delta$-Casorati curvatures for bi-slant submanifolds of
Kenmotsu space forms. J. Dyn. Syst. Geom. Theor., {\bf 17}(1) (2019), 39-50.

\bibitem{Lone_2019a} Lone, M.A., Shahid, M.H., V\^ilcu, G.E.: On Casorati
curvatures of submanifolds in pointwise Kenmotsu space forms. Math. Phys.
Anal. Geom., {\bf 22}(1) (2019), 1-14.

\bibitem{Nore_1986} Nore, T.: {Second fundamental form of a map}. Ann. Mat.
Pura Appl., {\bf 146} (1986), 281-310.

\bibitem{Neill_1966} O'Neill, B.: The fundamental equations of a submersion.
Michigan Math. J. {\bf 13}(4) (1966), 459-469.


\bibitem{OV_2011} Ons, B., Verstraelen, P.: Some geometrical comments on
vision and neurobiology: seeing Gauss and Gabor walking by, when looking
through the window of the Parma at Leuven in the company of Casorati.
Kragujevac J. Math., {\bf 35}(2) (2011), 317-325.

\bibitem{PLS} Polat, G., Lee, J.W., \c{S}ahin, B.: Optimal inequalities
involving Casorati curvatures along Riemannian maps and Riemannian
submersions for Sasakian space forms. J. Geom. Phys., {\bf 210} (2025), 1-15.

\bibitem{Sahin_2010} \c{S}ahin, B.: {Invariant and anti-invariant Riemannian
maps to K\"ahler manifolds}. Int. J. Geom. Methods Mod. Phys., {\bf 7}(3)
(2010), 337-355.


\bibitem{Sahin_book} \c{S}ahin, B.: {Riemannian Submersions, Riemannian Maps
in Hermitian Geometry, and Their Applications}. Elsevier, Academic Press,
(2017).

\bibitem{Siddiqui_2018} Siddiqui, A.N.: Optimal Casorati inequalities on
bi-slant submanifolds in generalized Sasakian space forms. Tamkang J. Math., 
{\bf 49}(3) (2018), 245-255.


\bibitem{SMM_2026} Singh, R., Meena, K., Meena, K.C.: {General Casorati inequalities and implications for Riemannian maps and Riemannian submersions}. J. Math. Anal. Appl., \textbf{558}(1) (2026), 1-31.

\bibitem{Singh_2026} Singh, R.: General Casorati inequalities for Riemannian submersions involving horizontal and verical Casorati curvatures and applications. arxiv preprint,
(2026), arXiv:2602.15804v1 [math.DG].


\bibitem{Tripathi_2017} Tripathi, M.M.: Inequalities for algebraic Casorati
curvature and their applications. Note Mat., {\bf 37}(1) (2017), 161-186.


\bibitem{Vilcu} V\^ilcu, G.E.: An optimal inequality for Lagrangian
submanifolds in complex space forms involving Casorati curvature. J. Math.
Anal. Appl., {\bf 465}(2) (2018), 1209-1222.

\bibitem{Yano} Yano, K., Kon, M.: Structure on Manifolds. World Scientific,
Singapore, (1984).


\bibitem{Zhang_Pan_Zhang} Zhang, L., Pan, X., Zhang, P.: Inequalities for
Casorati curvature of Lagrangian submanifolds in complex space forms. Adv.
Math. (China), {\bf 5} (2016), 767-777.

\bibitem{Zhang_Zhang} Zhang, P., Zhang, L.: Inequalities for Casorati
curvatures of submanifolds in real space forms. Adv. Geom., {\bf 16}(3)
(2016), 329-335.
\end{thebibliography}
\end{document}